\documentclass[12pt,reqno]{amsart}
\usepackage{amssymb, amsfonts, amsbsy, latexsym, epsfig, color}

\newtheorem{Thm}{Theorem}[section]

\newtheorem{corollary}[Thm]{Corollary}
\newtheorem{lemma}[Thm]{Lemma}

\newtheorem{proposition}[Thm]{Proposition}
\newtheorem{definition}[Thm]{Definition}

\newtheorem{example}[Thm]{Example}
\newtheorem{theorem}[Thm]{Theorem}

\newcommand{\bitem}{\begin{itemize}}
\newcommand{\eitem}{\end{itemize}}
\newcommand{\benum}{\begin{enumerate}}
\newcommand{\eenum}{\end{enumerate}}
\newcommand{\beq}{\begin{equation}}
\newcommand{\eeq}{\end{equation}}
\newcommand{\ip}[2]{\langle#1,#2\rangle}

\newcommand{\absip}[2]{| \langle#1,#2\rangle |}
\newcommand{\norm}[1]{\|#1\|}
\newcommand{\spann}{\mbox{\rm span}}

\newcommand{\Id}{\mbox{\rm Id}}
\newcommand{\tr}{\mbox{\rm tr}}

\def\RR{\mathbb{R}}

\def\FF{\mathbb{F}}

\def\cA{{\mathcal{A}}}

\def\cR{{\mathcal{R}}}

\def\cH{{\mathcal{H}}}

\def\SSn{\mathbb S}
\def\cHn{\mathcal H}
\def\cH{\mathcal{H}}

\newcommand{\gk}[1]{{\color{black}{#1}}}
\newcommand{\gkk}[1]{{\color{black}{#1}}}
\newcommand{\bgb}[1]{{\color{black}{#1}}}

\begin{document}

\title{A quantitative notion of redundancy for finite frames}

\author[B. G. Bodmann]{Bernhard G. Bodmann}
\address{Department of Mathematics, University of Houston, Houston, TX 77204-3008, USA}
\email{bgb@math.uh.edu}

\author[P. G. Casazza]{Peter G. Casazza}
\address{Department of Mathematics, University of Missouri, Columbia, MO 65211-4100, USA}
\email{pete@math.missouri.edu}

\author[G. Kutyniok]{Gitta Kutyniok}
\address{Institute of Mathematics, University of Osnabr\"uck, 49069 Osnabr\"uck, Germany}
\email{kutyniok@math.uni-osnabrueck.de}

\thanks{The first author was supported by NSF DMS 0807399, the second author by NSF DMS 0704216.
The third author would like to thank the Departments of Mathematics at the University of Missouri
and at the University of Houston for their hospitality and support during her visits, which enabled
completion of this work.}

\begin{abstract}
The objective of this paper is to improve the customary definition
of redundancy by providing quantitative measures in its place, which
we coin {\em upper and lower redundancies}, that match better with
an intuitive understanding of redundancy for finite frames \gk{ in a Hilbert
space}.
This motivates a carefully chosen list of desired properties for
upper and lower redundancies. The means to achieve these properties
is to consider the maximum and minimum of a redundancy function,
which is interesting in itself. The redundancy function is defined
on the sphere of the Hilbert space and measures the concentration of
frame vectors around each point. A complete characterization of
functions on the sphere which coincide with a redundancy function
for some frame is given. The upper and lower redundancies obtained
from this function are shown to satisfy all of the intuitively
desirable properties. In addition, the range of values they assume
is characterized.
\end{abstract}



\maketitle

\section{Introduction}

The theory of frames is nowadays a very well established field in which redundancy appears both
as a mathematical concept and as a methodology for signal processing. Frames ensure, for
instance, resilience against noise, quantization errors and
erasures in signal transmissions \cite{KC08}.
%They also provide methods for image processing such as geometric separation or image restoration.
Recently, the ability of redundant systems to provide sparse representations has been
extensively exploited \cite{BDE09}. Hence, it is fair to say that frames -- or redundant systems --
have become a standard notion in applied mathematics, computer science, and engineering.

Therefore one would expect that the `redundancy' of a frame, at least in finite dimensions, is a
well-explored concept.
However, to the best of the authors' knowledge, there does not exist a precise quantitative notion of redundancy other than the number
of frame vectors per dimension. From a scholarly point of view, this
is a rather unsatisfactory definition. It does not distinguish between the two toy examples of  frames
$\{e_1,e_1,e_1,e_2\}$ and $\{e_1,e_1,e_2,e_2\}$ in $\RR^2$ ($e_1$ and $e_2$ being the canonical orthonormal basis vectors), as one might wish,
nor does it provide much insight into properties of
the frame. Its main advantage is that for unit-norm tight frames it equals the value of
the frame bound.
Hence, although the idea of redundancy is the crucial property in \gk{various}
applications and thus the foundation of frame theory, a mathematically precise, meaningful definition is
missing.

\medskip

In this paper, we take a systematic approach to the problem of introducing a quantitative
notion of redundancy for finite frames \gk{in a Hilbert space $\cH$, say}, by first establishing a list
of desiderata that such a quantity should satisfy. We then propose a definition of a redundancy
function on the unit sphere \gk{in $\cH$,} which is shown to satisfy all
the postulated conditions, thereby immediately supplying us with a detailed list of
properties of a frame that its redundancy reveals.

\subsection{Review of Finite Frames}

We start by fixing our terminology while briefly reviewing the basic definitions
related to frames. \gk{Let} $\cHn$ denote an $n$-dimensional
real or complex Hilbert space. In this finite-dimensional situation, $\Phi = (\varphi_i)_{i=1}^N$
is called a {\em frame} for $\cHn$, if it is a -- typically, but not necessarily linearly dependent -- spanning
set. This definition is equivalent to ask for the existence of constants $0 < A \le B < \infty$
such that
\[
A\norm{x}^2 \leq \sum_{i=1}^N |\langle x, \varphi_i \rangle |^2 \leq B\norm{x}^2
\quad \mbox{for all } x \in \cHn.
\]
When $A$ is chosen as the largest possible value and $B$ as the smallest
for these inequalities to hold,
then we call them the {\em (optimal) frame constants}.
If $A$ and $B$ can be chosen as $A=B$, then the frame is called {\em $A$-tight}, and if
$A=B=1$ is possible, $\Phi$ is a {\em Parseval frame}. A frame is called
{\em equal-norm}, if there exists some $c>0$ such that $\|\varphi_i\|=c$ for all
$i=1,\ldots,N$, and it is {\em unit-norm} if $c=1$.

Apart from providing redundant expansions, frames can also serve as an analysis tool.
In fact, they allow the analysis of data by studying the associated {\em frame coefficients}
$(\langle x, \varphi_i \rangle)_{i=1}^N$, where the operator $T_\Phi$
defined by
$T_\Phi: \cHn \to \ell_2(\{1, 2, \dots, N\})$, $x \mapsto (\langle x,\varphi_i\rangle)_{i=1}^N$
is called the \emph{analysis operator}. The adjoint $T^*_\Phi$ of the analysis operator is typically
referred to as the {\em synthesis operator} and satisfies $T^*_\Phi((c_i)_{i=1}^N) = \sum_{i=1}^N c_i\varphi_i$.
The main operator associated with a frame, which provides a stable reconstruction process, is the
{\em frame operator}
\[
S_\Phi=T^*_\Phi T_\Phi : \cHn \to \cHn, \quad x \mapsto \sum_{i=1}^N \langle x,\varphi_i\rangle \varphi_i,
\]
a positive, self-adjoint invertible operator on $\cHn$. In the case of a Parseval frame, we have $S_\Phi=\Id_{\cHn}$.
In general, $S_\Phi$ allows reconstruction of a signal $x \in \cHn$ through the
reconstruction formula
\beq \label{eq:expansion}
x = \sum_{i=1}^N \langle x,S_\Phi^{-1} \varphi_i\rangle \varphi_i.
\eeq
The sequence $(S_\Phi^{-1} \varphi_i)_{i=1}^N$ which can be shown to form a frame itself, is often
referred to as the {\em canonical dual frame}.

We note that the choice of coefficients in the expansion \eqref{eq:expansion}
is generally not the only possible one. If the frame is linearly dependent --
which is typical in applications --
then there exist infinitely many choices of coefficients $(c_i)_{i=1}^N$ leading to expansions of
$x \in \cH$ by
\beq \label{eq:sparseexpansion}
x = \sum_{i=1}^N c_i \varphi_i.
\eeq
This fact, for instance, ensures resilience to erasures and noise. The particular choice of coefficients
displayed in \eqref{eq:expansion} is the smallest in $\ell_2$ norm \cite{Chr03}, hence contains the least energy.
\gkk{A different paradigm has recently received rapidly increasing attention, namely to choose
the coefficient sequence to be sparse in the sense of having only few non-zero entries, thereby allowing
data compression while preserving perfect recoverability (see, e.g., \cite{BDE09} and references
therein).}
%\gkk{DELETED PART HERE; ALSO DELETED REFERENCE DE03}

%Lately, the possibility inherent in the notion of redundancy to determine a {\em sparse expansion} of
%$x$ -- choosing the sequence $(c_i)_{i=1}^N$ to have only few non-zero entries -- has recently received
%rapidly increasing attention.
%Related with this goal is the notion of the {\em spark} of a matrix -- the
%smallest number of linearly dependent columns -- which was first introduced in \cite{DE03} as a measure
%of linear dependence. It was shown in \cite{DE03} that provided $(c_i)_{i=1}^N$ has less than
%$\spark(T_\Phi^*)/2$ nonzero components, the sparsest of all representations \eqref{eq:sparseexpansion}
%is necessarily unique. With slightly stronger conditions on $(c_i)_{i=1}^N$, the coefficients of this
%sparsest expansion can be determined by solving
%\[
%\min_{(c_i)_{i=1}^N} \norm{(c_i)_{i=1}^N}_1 \mbox{ such that } x = \sum_{i=1}^N c_i \varphi_i,
%\]
%see, for instance, also \cite{DE03}.

For a more extensive introduction to frame theory, we refer the interested reader to the books
\cite{Dau92,Mal98,Chr03} as well as to the survey papers \cite{KC07a,KC07b}.
%For a survey on sparsity, we refer to \cite{BDE09}.

\subsection{\gk{Problems with the Customary Notion of Redundancy}}
\label{subsec:vagueidea}

The previous subsection illustrated the fact that frame theory is entirely based on
the notion of redundancy. So far, the redundancy of a frame $\Phi = (\varphi_i)_{i=1}^N$
for an $n$-dimensional Hilbert space $\mathcal H$
was generally understood as
the quotient $\frac{N}{n}$, a customary, but somewhat crude measure as we will illustrate below.
If the finite frame is unit-norm and tight, then this
quotient coincides precisely with the frame bound, which in this case might
indeed serve as a redundancy measure. But let us consider the two frames $\Phi_{1,s}$
and $\Phi_2$ for $\cHn$ defined by
\beq \label{eq:example1}
\Phi_{1,s} = \{e_1, \ldots, e_1, e_2, e_3, \ldots, e_n\}, \quad \mbox{where $e_1$ occurs $s$ times},
\eeq
and
\beq \label{eq:example2}
\Phi_2 = \{e_1, e_1, e_2, e_2, e_3, e_3, \ldots, e_n, e_n\},
\eeq
with $\{e_1, \ldots, e_n\}$ being an orthonormal basis for $\cHn$. If $s=n+1$, then
the crude, customary measure of redundancy  coincides for $\Phi_{1,s}$ and $\Phi_2$. However, intuitively
the redundancy of $\Phi_{1,s}$ seems to be very localized, whereas the redundancy of
$\Phi_2$ seems to be quite uniform. The fact that $\Phi_2$ can be split into two spanning
sets, but $\Phi_{1,s}$ cannot, gives further support to this intuition. Also, the frame $\Phi_2$ is
robust with respect to any one erasure, whereas $\Phi_{1,s}$ does not have this property.
Neither of these facts can be read from the customary redundancy measure, which makes it rather
unsatisfactory.

Concerning infinite-dimensional Hilbert spaces, research has already progressed, and we refer to the recent
publication \cite{BL07} (see also \cite{BCL09}). In this paper, the authors provide
a meaningful quantitative notion of redundancy which applies to general
infinite frames. In their work, redundancy
is defined as the reciprocal of a so-called frame measure function, which is a
function of certain averages of inner products of frame elements with their
corresponding dual frame elements.  More recently, in \cite{BCL09}, it is shown that
$\ell_1$-localized frames satisfy
several properties intuitively linked to redundancy such as that any frame with
redundancy \gk{greater} than one should contain in it a frame with redundancy arbitrarily
close to one, the redundancy of any frame for the whole space should be greater than
or equal to one, and that the redundancy of a Riesz basis should be exactly one,
were proven for this notion. However, as stated, these notions of redundancy only apply to infinite
frames.

\subsection{\gk{An Intuition-Driven Approach to Redundancy}}
\label{subsec:intuition}

Concluding from the previous subsection, there does not exist a satisfactory notion
of redundancy for finite frames, which forces us to first build up intuition on
what we expect of such a notion. In order to properly define a quantitative notion
of redundancy, we will agree on a list of desiderata that our notion is required
to satisfy.

Inspired by the two examples \eqref{eq:example1} and \eqref{eq:example2}, which
our notion of redundancy shall certainly distinguish, we realize that the local concentration
of frame vectors should play an essential role. Hence a local
notion of redundancy is desirable. Based on this, suitable global redundancy measures
should be the minimal and maximal possible local redundancy attained. This philosophy
also coincides with the philosophy of density considerations upon which the notion of redundancy
for infinite frames is built \cite{BL07,BCL09}, since lower and upper densities have
been studied multiple times giving suitable measures for local concentrations
(see \cite{Hei07}). Coming back to \eqref{eq:example1} and \eqref{eq:example2},
ideally, the upper redundancy of \eqref{eq:example1} should be $s$ and the lower
$1$, whereas the upper and lower redundancies of \eqref{eq:example2} should
coincide and equal $2$.
More generally, if a frame consists of orthonormal basis vectors
which are individually repeated several times, then the lower redundancy should
be the smallest number of repetitions and the upper redundancy the largest.
\gk{This should still hold true if the single frame vectors are arbitrarily scaled,
since resilience against erasures as one main aspect of redundancy should intuitively be invariant
under this operation.}

But even more, for general frames, redundancy should give us information, for instance, about
orthogonality and tightness of the frame, about the maximal number of spanning
sets and the minimal number of linearly independent sets our frame can be
divided into, and about robustness with respect to erasures. Ideally,
in the case of a unit-norm tight frame, upper and lower redundancy
should coincide and equal the customary measure of redundancy, which seems the appropriate
description for this very particular class of frames.

\subsection{Desiderata}
\label{subsec:desiderata}

Summarizing and analyzing the requirements we have discussed, we state the following list of
desired properties for an upper redundancy  $\cR^+_\Phi$ and a lower redundancy $\cR^-_\Phi$
of a frame $\Phi = (\varphi_i)_{i=1}^N$ for an $n$-dimensional real or complex Hilbert space $\cHn$.

\renewcommand{\labelenumi}{{\rm [D\arabic{enumi}]}}

\begin{enumerate}
\item\label{DEE} {\em Generalization.} If $\Phi$ is an equal-norm Parseval frame, then
in this special case the customary notion of redundancy shall be attained, i.e., $\cR^-_\Phi = \cR^+_\Phi = \frac{N}{n}$.
\item\label{DNyquist} {\em Nyquist Property.} The condition $\cR^-_\Phi = \cR^+_\Phi$ shall characterize
tightness of a normalized version of $\Phi$, thereby supporting the intuition that upper and lower
redundancy being different implies `non-uniformity' of the frame. In particular, $\cR^-_\Phi =
\cR^+_\Phi = 1$ shall be equivalent to orthogonality as the `limit-case'.
\item\label{Duplow} {\em Upper and Lower Redundancy.} Upper and lower redundancy shall be
`naturally' related by $0 < \cR^-_\Phi \le \cR^+_\Phi < \infty$.
\item\label{Dadditiv} {\em Additivity.} Upper and lower redundancy shall be subadditive and superadditive,
respectively, with respect to unions of frames. They shall be additive provided that the redundancy
is uniform, i.e., $\cR^-_\Phi=\cR^+_\Phi$.
\item\label{DInvariance} {\em Invariance.} Redundancy shall be invariant under the action of a unitary
operator on the frame vectors, under scaling of single frame vectors, as well as under permutation,
since intuitively all these actions should have no effect on, for instance, robustness against erasures,
\gk{which is one property redundancy shall intuitively measure.}
\item\label{Daverob} {\em Spanning Sets.} The lower redundancy shall measure the maximal number of
spanning sets of which the frame consists. This immediately implies that the lower
redundancy is a measure for robustness of the frame against erasures in the sense that any set of
a particular number of vectors can be deleted yet leave a frame.
\item\label{Dmaxrob} {\em Linearly Independent Sets.} The upper redundancy shall measure the minimal
number of linearly independent sets of which the frame consists.
%\gkk{DELETED PART HERE.}
%Notice that this is an upper bound for the spark of a
%frame, which is the minimal number of elements forming a linearly dependent subset of the frame.
\end{enumerate}

It is straightforward to verify that for the special type of frames consisting of orthonormal basis vectors, each
repeated a certain number of times, the upper and lower redundancies given by the maximal or minimal
number of repetitions satisfy these conditions. The challenge is now to extend this definition to all
frames in such a way that the properties are preserved.

\renewcommand{\labelenumi}{{\rm (\roman{enumi})}}

\subsection{Contribution of this paper}

The contribution of this paper is two-fold. Firstly, we introduce a list of
desiderata that a rationally defined quantitative notion of redundancy for a finite frame
shall satisfy. Secondly, we propose a notion of upper and lower redundancy
which will be shown to, in particular, satisfy all properties advocated
in those desiderata.

%************************************************************************************

\section{Defining Redundancy and Main Result}

\subsection{Definitions}

As explained before, we first introduce a local redundancy, which
encodes the concentration of frame vectors around one point. Since the
norms of the frame vectors do not matter for concentration, we normalize
the given frame and also consider only points on the unit sphere $\SSn=\{x \in \cHn: \|x\|=1\}$
in $\cHn$. Hence another way to view local redundancy is by considering
it as some sort of density function on the sphere.

We now define a notion of local redundancy. \gk{For this, we remark that} throughout the paper, we
let \gk{$\langle y \rangle$ denote the span of some $y \in \cHn$ and $P_{\langle y \rangle}$
the orthogonal projection onto $\langle y \rangle$.}

\begin{definition}
Let $\Phi = (\varphi_i)_{i=1}^N$ be a frame for a finite-dimensional real or complex Hilbert space $\cHn$. For
each $x \in \SSn$, the {\em redundancy function}
$\cR_\Phi : \SSn \to \RR^+$ is defined by
\[
\cR_\Phi(x) = \sum_{i=1}^N \|P_{\langle \varphi_i \rangle} (x)\|^2
%= \sum_{i=1}^N|\langle x,\frac{\varphi_i}{\|\varphi_i\|}\rangle|^2
.
\]
\end{definition}

We might think about the function $\mathcal R_\Phi$ as a redundancy pattern on the sphere,
which measures redundancy at each single point. Also notice that this notion is reminiscent
of the fusion frame condition \cite{CKL08}, here for rank-one projections.

The following observation is rather trivial, but useful.

\begin{lemma}
\label{lemma:maxmin}
If $\Phi = (\varphi_i)_{i=1}^N$ is a frame for a finite-dimensional real or complex Hilbert space $\cHn$, then
the redundancy function $\cR_\Phi$ assumes its maximum and its minimum on the unit
sphere in $\mathcal H$.
%Let $\Phi = (\varphi_i)_{i=1}^N$ be a frame in $\cHn$.
%Then the following conditions hold.
%\bitem
%\item[{\rm (i)}] The function $x \mapsto \cR_\Phi(x)$, $\SS^n \to \RR^+_0$ is continuous.
%\item[{\rm (ii)}] The values $\max_{x \in \SS^n} \cR_\Phi(x)$ and $\min_{x \in \SS^n} \cR_\Phi(x)$
%do exist.
%\eitem
\end{lemma}

\begin{proof}
By definition, the function $\mathcal R_\Phi$ is continuous. Moreover, since the unit sphere in finite dimensions is compact,
the function attains its extrema.
\end{proof}

This consideration allows us to define the maximal and minimal value the redundancy
function attains as upper and lower redundancy.

\begin{definition} \label{def:upplowred}
Let $\Phi = (\varphi_i)_{i=1}^N$ be a frame for a finite-dimensional real or complex Hilbert space $\cHn$.
Then the {\em upper redundancy of $\Phi$} is defined by
\[
\cR^+_\Phi = \max_{x \in \SSn} \cR_\Phi(x)
\]
and the {\em lower redundancy of $\Phi$} by
\[
\cR^-_\Phi = \min_{x \in \SSn} \cR_\Phi(x).
\]
Moreover, $\Phi$ has a {\em uniform redundancy}, if
\[
\cR^-_\Phi = \cR^+_\Phi.
\]
\end{definition}

This notion of redundancy hence equals the upper and lower frame bound of the
normalized frame. Surprisingly, this notion will be shown to satisfy
all desiderata we phrased.

\subsection{Main Result}

With the previously defined quantitative notion of upper and lower redundancy, we can
now verify the desired properties from Subsection \ref{subsec:desiderata} explicitly in the
following theorem, whose proof will be given in Section \ref{sec:proof}.

\begin{theorem}
\label{theo:desiderata}
Let $\Phi = (\varphi_i)_{i=1}^N$ be a frame for an $n$-dimensional real or complex Hilbert space   $\cHn$.
\begin{enumerate}
\item[{\rm [D1]}] {\em Generalization.} If $\Phi$ is an equal-norm Parseval frame, then
\[
\cR^-_\Phi = \cR^+_\Phi = \frac{N}{n}.
\]
\item[{\rm [D2]}] {\em Nyquist Property.} The following conditions are equivalent:
\bitem
\item[{\rm (i)}] We have $\cR^-_\Phi = \cR^+_\Phi$.
\item[{\rm (ii)}] The normalized version of $\Phi$ is tight.
\eitem
Also the following conditions are equivalent.
\bitem
\item[{\rm (i')}] We have $\cR^-_\Phi = \cR^+_\Phi = 1$.
\item[{\rm (ii')}] $\Phi$ is orthogonal.
\eitem
\item[{\rm [D3]}] {\em Upper and Lower Redundancy.} We have
\[
0 < \cR^-_\Phi \le \cR^+_\Phi < \infty.
\]
\item[{\rm [D4]}] {\em Additivity.} For each orthonormal basis $(e_i)_{i=1}^n$,
\[
\cR^\pm_{\Phi\cup(e_i)_{i=1}^n} = \cR^\pm_\Phi + 1.
\]
Moreover, for each frame $\Phi'$ in $\cHn$,
\[
\cR^-_{\Phi\cup\Phi'} \ge \cR^-_\Phi + \cR^-_{\Phi'}
\quad \mbox{and} \quad
\cR^+_{\Phi\cup\Phi'} \le \cR^+_\Phi + \cR^+_{\Phi'}.
\]
In particular, if $\Phi$ and $\Phi'$ have uniform redundancy, then
\[
\cR^-_{\Phi\cup\Phi'} = \cR_\Phi + \cR_{\Phi'} = \cR^+_{\Phi\cup\Phi'}.
\]
\item[{\rm [D5]}] {\em Invariance.}
Redundancy is invariant under application of a unitary operator $U$ on $\cHn$, i.e.,
\[
\cR^\pm_{U(\Phi)} = \cR^\pm_{\Phi},
\]
under scaling of the frame vectors, i.e.,
\[
\cR^\pm_{(c_i \varphi_i)_{i=1}^N} = \cR^\pm_{\Phi}, \quad c_i \mbox{ scalars},
\]
and under permutations, i.e.,
\[
\cR^\pm_{(\varphi_{\pi(i)})_{i=1}^N} = \cR^\pm_{\Phi}, \quad \pi \in S_{\{1,\ldots,N\}},
\]
\item[{\rm [D6]}] {\em Spanning Sets.} $\Phi$ contains
$\lfloor \cR^-_\Phi\rfloor$ disjoint spanning sets.
In particular, any set of $\lfloor\cR^-_\Phi\rfloor-1$ vectors can be deleted yet leave a frame.
\item[{\rm [D7]}] {\em Linearly Independent Sets.} \bgb{If
$\Phi$ does not contain any zero vectors, then it}
can be partitioned into $\lceil \cR^+_\Phi\rceil$ linearly independent sets.
%\gkk{DELETED PART HERE.}
%In particular, we have $\spark(T_\Phi^*) \le \lceil \cR^+_\Phi\rceil$.
\end{enumerate}
\end{theorem}

\subsection{Examples}

Let us now analyze the running examples \eqref{eq:example1} and \eqref{eq:example2}
from Subsection \ref{subsec:desiderata}, as well as introduce and analyze one
additional frame.  We will show that the upper and lower
redundancies precisely equal those values, which we intuitively anticipated a
reasonable notion to attain. We will further exploit Theorem \ref{theo:desiderata}
to derive additional information about the frames.

\begin{example}
\label{example:F1}
{\rm $\Phi_{1,s}$ satisfies
\[
\cR^-_{\Phi_{1,s}} = 1 \quad \mbox{and} \quad \cR^+_{\Phi_{1,s}} = s.
\]
This can be seen as follows. By definition,
\[
\cR_{\Phi_{1,s}}(x) = \sum_{i=1}^N \|P_{\langle \varphi_i \rangle}(x)\|^2 = s
\absip{x}{e_1}^2 + \sum_{i=2}^{n} \absip{x}{e_i}^2.
\]
Hence $\cR_{\Phi_{1,s}}(e_1) = s$. For $x \neq e_1$, letting $c_i =
\ip{x}{e_i}$,
\[
\cR_{\Phi_{1,s}}(x) = s c_1^2 + \sum_{i=2}^{n} c_2^2 = (s-1) c_1^2 + 1 <
(s-1) + 1 = s.
\]
This implies $\cR^+_{\Phi_{1,s}} = s$. Moreover,
\[
\cR_{\Phi_{1,s}}(e_2) = 1 \le \cR_{\Phi_{1,s}}(x) \quad \mbox{for all } x \neq e_2,
\]
which implies $\cR^-_{\Phi_{1,s}} = 1$.

Exploiting [D2], the frame $\Phi_{1,s}$ is neither orthogonal nor is it tight. [D6] tells us that
$\Phi_{1,s}$ can be split into $1$ spanning set, which is indeed the maximal number,
since, for instance, $e_2$ occurs one time and is orthogonal to all other elements
from the frame. Concluding from [D7], $\Phi_{1,s}$ can be partitioned into $s$ linearly
independent sets, which can be chosen as $\{e_1\}$ $s-1$ times and $\{e_1, \ldots, e_n\}$.
It is also evident that this is the minimal possible number, since the $s$ vectors $e_1$
need to be placed into separate linearly independent sets.
Naturally, the frame $\Phi_{1,s}$ can be normalized to become a tight frame. The upper and lower
redundancies however remain the same, since only the system of normalized vectors is considered.
%The interpretation of redundancy as a measure for the ability to partition into linearly independent
%sets, spanning sets, and as a measure for erasure resilience, supports independence of the norms
%of the vectors.
}
\end{example}

\begin{example}
\label{example:F2}
{\rm $\Phi_2$ possesses a uniform redundancy.  More precisely,
\[
\cR^-_{\Phi_2} = 2 \quad \mbox{and} \quad \cR^+_{\Phi_2} = 2.
\]
This follows from
\[
\cR_{\Phi_2}(x) = \sum_{i=1}^N \|P_{\langle \varphi_i \rangle}(x)\|^2 = 2
\sum_{i=1}^{n} \absip{x}{e_i}^2 = 2,
\]
and taking the max and the min over the sphere.

Notice that $\Phi_2$ is a $2$-tight frame. Hence the uniform redundancy
coincides with the customary notion of redundancy as the quotient $(2n)/n = 2$ by [D1].
Further, by [D6] and [D7], $\Phi_2$ can be
partitioned into 2 spanning sets and also into 2 linearly independent sets. Those
partitions can here in fact be chosen to be the same, more precisely, can be chosen
to be the two orthonormal bases of which $\Phi_2$ is composed.}
\end{example}

\begin{example}
\label{example:F3}
{\rm We add a third example in which the frame is not merely composed of vectors from the unit basis
$\{e_1, \ldots, e_n\}$. Letting $0 < \varepsilon < 1$, we choose $\Phi_3 = (\varphi_i)_{i=1}^N$
as
\[
\varphi_i = \left\{\begin{array}{rcl}
e_1 & : & i=1,\\
\sqrt{1-\varepsilon^2} e_1 + \varepsilon e_i & : & i \neq 1.
\end{array} \right.
\]
This frame is strongly concentrated around the vector $e_1$. We first observe that
\[
\cR_{\Phi_3}(e_1) = \sum_{i=1}^N \|P_{\langle \varphi_i \rangle}(e_1)\|^2
= 1 + \sum_{i=2}^N \absip{e_1}{\sqrt{1-\varepsilon^2} e_1 + \varepsilon e_i}^2
= 1 + (N-1)(1-\varepsilon^2).
\]
However, this is not the maximum, which is in fact attained at the average point of the
frame vectors. But in order to avoid clouding the intuition by technical details,
we omit this analysis, and observe that
\[
1 + (N-1)(1-\varepsilon^2) \le \cR_{\Phi_3}^+ < N.
\]
Since
\[
\cR_{\Phi_3}(e_2) = \sum_{i=1}^N \|P_{\langle \varphi_i \rangle}(e_2)\|^2
= \sum_{i=2}^N \absip{e_2}{\sqrt{1-\varepsilon^2} e_1 + \varepsilon e_i}^2
= \varepsilon^2,
\]
we can conclude similarly, that
\[
0 < \cR_{\Phi_3}^- \le \varepsilon^2.
\]}
\end{example}
The frame $\Phi_3$ shows that the new redundancy notion gives little information
near the extreme cases:  $\cR^- \approx 0$ and $\cR^+ \approx N$,
but becomes increasingly more accurate as $\cR^-$ and $\cR^+$ become closer to one another
(cf. the discussion after Theorem \ref{lemma:rangeR}).
By [D2], the frame $\Phi_3$ is not orthogonal, nor is it tight. [D6] is not applicable
for this frame, since $\lfloor \cR_{\Phi_3}^- \rfloor = 0$ although there does exist a
partition into one spanning set. Hence, [D6] is not sharp, but becomes telling, if
$\cR^- \ge 2$. Now, [D7] implies that this frame can be partitioned into $N-1$ linearly
independent sets. Again, we see that we can do better than this by merely taking the
whole frame which happens to be linearly independent. As before, we observe that [D7]
is not sharp for large values of $\cR^+$.  However, these become increasingly accurate
as $\cR^-$ and $\cR^+$ approach $N/n$ (again we refer to the discussion following Theorem
\ref{lemma:rangeR}).

%************************************************************************************

\section{Characterization of Redundancy Functions}

%\section{The Redundancy Function}

%\subsection{Characterization of Redundancy Functions}

As we just saw, the notion of redundancy is built upon the notion of a redundancy
function. Hence to analyze the notion of redundancy more closely, we will investigate
characteristic properties of redundancy functions.

Interestingly, the redundancy function itself is a useful object to exploit.
Assume we use a frame $\Phi$ to encode a vector $x$
into its frame coefficients to prevent data loss if some of the coefficients are erased (lost or impractically delayed).
Indeed, any input vector can still be perfectly recovered if the set
of frame vectors belonging to coefficients, which are not erased, form a frame $\Phi'$, or equivalently,
if the associated redundancy function $\mathcal R_{\Phi'}$ is strictly positive.
However, $\mathcal R_{\Phi'}$ contains useful information even if this is not the case:
For any input vector $x$ and any residual set $\Phi'$,
the projection of $x$ onto the orthogonal complement of the zero set of $\mathcal R_{\Phi'}$ can be recovered.
In practice, the input and the erasures are typically random, and only some information about their
distribution is known. To achieve a small distortion of the transmitted vector, it is then desirable to
choose the frame in such a way that the input vectors are concentrated near
the orthogonal complement of the (random) zero set of $\mathcal R_{\Phi'}$.
%on the sphere is known. Then this situation can be interpreted as having knowledge of a
%required redundancy function. A natural question is whether there does exist a frame which
%is associated with this redundancy function. In a certain sense this would be the optimal
%frame to transmit these data. Providing an answer to this question will be our main objective
%in this section.
%Assume we are given a random input vector $X$ which assumes values on the sphere $\mathbb S^n$
%with a certain probability distribution and want to use a frame to encode this vector for
%transmissions through an erasure channel. For a given outcome $x\in \mathbb S^n$ of the input vector,
%having redundancy $R(x)\ge m+1$ implies that at least $m$
%erased frame coefficients of $x$ can be recovered.
%Further assuming that $m$ erasures occur with high probability and that a certain (average) redundancy
%budget is given, meaning the integral of $\mathcal R_\Phi$ over the sphere is fixed,
%then one might want to maximize the probability of recovering
%the random input, $P(\mathcal R_\Phi(X) \ge m+1)$.
This is a new type of
frame design problem arising from the redundancy function, which we will
investigate in detail elsewhere.

The first observation we make in this context is that  a frame is not uniquely
specified by its redundancy function,
since we can scale the single frame vectors arbitrarily, yet the frame is still
associated with the same normalized frame. Thus, searching for a frame
which possesses a predefined redundancy function is in fact searching for the
following equivalence class.

\begin{proposition}
Let $\FF$ be the set of frames for  a finite-dimensional real or complex Hilbert space $\cHn$. Then the relation $\sim$ on $\FF$ defined by
\[
\Phi \sim \Psi \quad : \Longleftrightarrow \quad \cR_\Phi = \cR_\Psi
\]
is an equivalence relation on $\FF$.
\end{proposition}

\begin{proof}
All three conditions, reflexivity, symmetry, and transitivity are immediate by definition of
the relation.
\end{proof}

The just introduced equivalence relation can be described in a different way by linking the redundancy
function with the associated frame operator. For this, we require the following notion: For a
frame $\Phi = (\varphi_i)_{i=1}^N$ in $\cHn$, we let $\tilde{S}_\Phi$ denote the frame operator with
respect to the normalized version of $\Phi$, i.e.,
\[
\tilde{S}_\Phi(x) = \sum_{i=1}^N P_{\langle \varphi_i \rangle}.
\]
Further, we denote the associated quadratic form by
\[
\mathcal Q_\Phi(x) = \langle \tilde S_\Phi x, x\rangle,
\]
and note that$\mathcal Q_\Phi$ extends $\mathcal R_\Phi$ to all $x \in \mathcal H$.

We recall that indeed, there is a one-one correspondence
between positive (semi-)definite operators and quadratic forms.

\begin{theorem}{\protect\cite[Theorem 3.5]{Wei80}} \label{theo:friedrich}
Each positive (semi-)definite, bounded operator $A$
on a real or complex Hilbert space $\mathcal H$ is uniquely determined by the associated quadratic
form $\mathcal Q_A(x)=\langle Ax,x\rangle$, $x \in \mathcal H$.
\end{theorem}

The essence of the proof is the so-called polarization identity. For
real Hilbert spaces, we have
$$
  \langle A x, y\rangle = \frac 1 4 (\mathcal Q_A(x+y) -\mathcal Q_A(x-y))
  $$
and for the complex case
$$
 \langle A x, y\rangle = \frac 1 4 (\mathcal Q_A(x+y) -\mathcal Q_A(x-y))
   + \frac{i}{4} (\mathcal Q_A(x+iy) -\mathcal Q_A(x-iy)) \, .
$$
We use this fact to establish when two frames belong to the same equivalence class.

\begin{corollary}
\label{lemma:RandS}
If $\Phi, \Psi$ are two frames for a finite-dimensional Hilbert space $\cHn$, then the following conditions are equivalent.
\bitem
\item[{\rm (i)}] $\cR_\Phi = \cR_\Psi$ on $\SSn$.
\item[{\rm (ii)}] $\tilde{S}_\Phi = \tilde{S}_\Psi$ on $\cHn$.
\eitem
\end{corollary}

\begin{proof}
The redundancy function of a frame $\Upsilon$ extends to all $x \in \mathcal H$
by the quadratic scaling
\[
\mathcal Q_\Upsilon(x)=\|x\|^2 \cdot \mathcal R_\Upsilon(x/\|x\|),
\]
which defines a quadratic form and thus is equal to
$\mathcal Q_\Upsilon(x)=\langle \tilde S_\Upsilon x,x\rangle$. Since the quadratic
form $\mathcal Q_\Upsilon$ and the operator $\tilde S_\Upsilon$ are in one-to-one correspondence
by Theorem \ref{theo:friedrich},
equality of the redundancy function for two frames $\Phi$ and $\Psi$ is equivalent
to the normalized
frame operators being identical.
%
%The first part follows immediately from the definition of $\cR_\Phi$ and $\tilde{S}_\Phi$.
%The equivalence of (i) and (ii) then follows from here.
\end{proof}

It follows from Corollary \ref{lemma:RandS} that equivalent frames
$\Phi$ and $\Psi$, i.e., $\cR_\Phi \sim \cR_\Psi$,
must have the same
number of non-zero frame vectors.  That is, the number of non-zero frame vectors is the sum of the
eigenvalues of the {\em equal} frame operators $\tilde{S}_\Phi = \tilde{S}_\Psi$.

Since tight frames are in some sense the most natural generalization of orthonormal
bases, one might ask whether each equivalence class contains at least one tight frame.
It is easily seen that, for each $N\ge n$,
one of the equivalence classes contains all the unit norm tight frames with $N$ vectors
plus non-zero multiples of their frame vectors.  i.e.  $(c_i \Phi_i)_{i=1}^N$, with $c_i \not= 0$.
Other classes {\em may} contain tight frames.  For example, the equivalence class of the
frame $\Phi_{1,s}$ contains the Parseval frame
\[ \{ s^{-1/2} e_1, \ldots, s^{-1/2} e_1,
e_2,\ldots, e_n\}, \quad \mbox{where $e_1$ occurs $s$ times}.\]
In general, an equivalence class need not contain any tight frames at all.  For example, consider
the frame $\Phi_3$. This is a unit-norm linearly independent set, whose frame operator is not a multiple of
the identity, but the sum of its eigenvalues is
$n$.  Now assume towards a contradiction that there exists a tight frame $\Psi = (\psi_i)_{i=1}^m$
in its equivalence class. Then -- as just mentioned -- we must have that $m=n$, hence it must
be an equal-norm orthogonal set. This implies $\cR_{\Psi} = {\rm Id}_{\cHn} \not= \cR_{\Phi_3}$,
a contradiction.

\medskip

Given a positive self-adjoint rank-$n$ operator $S$ on
$\cHn$, we next characterize when it coincides with the frame operator
of a normalized frame.
%for every integer $N\ge n$, there is an equal-norm sequence $(\varphi_i)_{i=1}^N$
%in $\cHn$ which has $S$ as its frame operator \cite[Corollary 3.2]{CT09}, the frame operators $\tilde{S}_\Phi$
%associated with normalized frames are more special.
%{\bf REVERTED TO 10/8 BUT PLEASE EVALUATE THIS SENTENCE} The following proposition provides a
%complete characterization of those.{\bf END REVERTED}

\begin{proposition}
\label{prop:operatorcharac}
Let $T$ be a positive, invertible operator on a real or complex Hilbert space $\cHn$ with eigenvalues $\lambda_1, \ldots, \lambda_n$.
Then the following conditions are equivalent.
\bitem
\item[{\rm (i)}] There exists a frame $\Phi = (\varphi_i)_{i=1}^N$ with $\varphi_i \neq 0$ for all $i$
such that $T=\tilde{S}_\Phi$.
\item[{\rm (ii)}] There exists some $N \in \mathbb N$, $N \ge n$ such that
\[
\sum_{i=1}^n \lambda_i = N.
\]
\eitem
\end{proposition}

\begin{proof}
(i) $\Rightarrow$ (ii). By (i), $T$ is the frame operator associated
with the normalized version of $\Phi$.  Hence the trace of $T$ \gk{satisfies}
\[
\gk{\tr[T] = }N=\sum_{i=1}^N \left\|\varphi_i/\norm{\varphi_i}\right\|^2 = \sum_{j=1}^n \lambda_j \, ,
\]
and the number of frame vectors has to satisfy $N\ge n$ because $\Phi$ is spanning.

(ii) $\Rightarrow$ (i). Let $(e_i)_{i=1}^n$ be the orthonormal eigenbasis of $T$ such that
\[
Te_i = \lambda_i e_i, \quad i=1,\ldots,n \, ,
\]
and assume the sum of the eigenvalues is an integer $N\ge n$.
By a  result from \cite{DFKLOW,CL09}, there exists an equal-norm frame
$\Phi=(\varphi_i)_{i=1}^N$ having $T$ as its frame operator, with
$\|\varphi_i\|=1$ for all $i\in \{1,2, \dots n\}$. This implies
$S_\Phi=\tilde S_\Phi$, thus $T=\tilde S_\Phi$ as required.
\end{proof}

We remark that Proposition \ref{prop:operatorcharac} is constructive, because
the inductive proof in \cite{DFKLOW,CL09}  provides such a frame.

To prepare the characterization of redundancy functions further, we recall
Gleason's theorem.
\begin{theorem}\cite{Gle57}  \label{theo:Gleason}
Let $\mathcal H$ be a Hilbert space of dimension $n\ge 3$, and let $g: \mathbb S \to \mathbb R^+_0$
be chosen such that
\[
\sum_{i=1}^n g(e_n) = 1
\]
for any orthonormal basis $\{e_1, e_2, \dots e_n\}$. Then there exists a trace-normalized positive definite
operator $T$ such that $\langle Tx,x\rangle = g(x)$ for all $x \in \mathbb S^n$.
\end{theorem}

We need a slight generalization of this theorem.

\begin{corollary}\label{cor:Gleason}
Let $\mathcal H$ be a Hilbert space of dimension $n\ge 3$, let $N > 0$, and let
$g: \mathbb S \to \mathbb R^+_0$ be chosen such that
\[
\sum_{i=1}^n g(e_n) = N
\]
for any orthonormal basis $\{e_1, e_2, \dots e_n\}$. Then there exists a positive definite
operator $T$ with \gk{$\tr[T]=N$} such that $\langle Tx,x\rangle = g(x)$ for all $x \in \mathbb S$.
Moreover,  $g$ is strictly positive if and only if $T$ is invertible.
\end{corollary}

\begin{proof}
The first part is a simple scaling argument. The second part of this corollary follows from the
first one because the minimum of $g$ is the smallest eigenvalue of $T$.
\end{proof}

We are now ready to state and prove the main result of this section which provides a
complete characterization of all functions \gk{on} the sphere which are redundancy functions
of an equivalence class of frames.

\begin{theorem}
\label{theo:equivalenceR}
Let $f : \SSn \to \RR_0^+$, $\mathcal H$ be an $n$-dimensional real or complex Hilbert space
with $n\ge 3$,
and let $q$ be the extension of $f$ to $\mathcal H$ given by $q(0)=0$
and $q(x)=\|x\|^2  f(x/\|x\|)$ for any $x \neq 0$. Let $\omega$ denote
the probability measure on the unit sphere which is invariant under
all unitaries.
Then the following conditions are equivalent.
\bitem
\item[{\rm (i)}] There exists a frame $\Phi$ for $\cHn$ such that
\[
f(x) = \cR_\Phi(x) \quad \mbox{for all } x \in \SSn.
\]
\item[{\rm (ii)}] The function $f$ is strictly positive on $\mathbb S$, its extension
$q$
satisfies the parallelogram identity
\[
q(x+y) + q(x-y) = 2(q(x)+q(y)) \quad \mbox{for all } x,y \in \SSn
\]
and $f$ integrates to
$$
  \int_{\SSn} f(x)d\omega(x)= N/n \,
$$
with some integer $N\ge n$.
\item[{\rm (iii)}]
The function $f$ is strictly positive on $\mathbb S$ and there exists an integer $N\ge n$ such that
for any ortho\-normal basis $\{e_1, e_2, \dots e_n\}$,
\[
\sum_{i=1}^n f(e_i) = N \, .
\]
\eitem
Also the following conditions are equivalent.
\bitem
\item[{\rm (i')}] There exists a unit-norm tight frame $\Phi$ for $\cHn$ such that
\[
f(x) = \cR_\Phi(x) \quad \mbox{for all } x \in \SSn.
\]
\item[{\rm (ii')}] There exists some integer $N \ge n$ such that
\[
f(x)=N/n \quad \mbox{for all } x \in \SSn.
\]
\eitem
\end{theorem}

\begin{proof}
We first focus on the equivalence of (i) and (ii). For this, notice that, by Corollary~\ref{lemma:RandS},
(i) is equivalent to the existence of a frame $\Phi$ for $\cHn$ such that $f(x) =
\norm{\tilde{S}_\Phi^{1/2} x}^2$ for all $x \in \SSn$, where $\tilde{S}_\Phi^{1/2}$ is positive
and invertible. Since, by \cite{CT09}, each positive, invertible operator is the
frame operator of an equal-norm frame, (i) is equivalent to the existence of a positive, invertible operator $S$
satisfying   $f(x) = \norm{S^{1/2} x}^2$ for all $x \in \SSn$. Hence (i) is equivalent to $f$ defining a new norm
\[
\norm{S^{1/2} x}^2 = f(x) =: |||x|||^2
\]
on $\mathcal H$ which, by Lemma \ref{lemma:RandS}, satisfies
\[
|||x|||^2 =  \ip{S x}{x}.
\]
Since this defines a new inner product on $\cHn$ by $(x,y) = \ip{S x}{y}$, (i) is equivalent
to the fact that the norm defined by $f$ is induced by an inner product. Hence, by the Jordan-von Neumann theorem
\cite{JN35}, condition (i) is equivalent to the parallelogram identity. Moreover,
the operator associated with the quadratic form $q$
can be written  as a sum of orthogonal projection operators
if and only if its trace is a positive integer. This amounts to
\[
\int_{\mathbb S}  f(x)d\omega(x) =N/n, \quad  \mbox{for some positive integer }N\, ,
\]
see, for example, the proof in \cite[Proposition 3.2]{BBCE09}.
Finally we observe that $f$ is strictly positive if and only if $N\ge n$, because otherwise
the sum of projections would not be invertible and thus would not yield a frame operator.

To verify the equivalence of (i) and (iii), we note that given a frame $\Phi$ we can always remove
vanishing vectors from it without changing $\mathcal R_\Phi$. Since $\mathcal R_\Phi$ extends to the quadratic
form of the normalized frame operator $\tilde S_\Phi$,  its trace satisfies
$\mathop{\mathrm{tr}}[\tilde S_\Phi]= N$, $N$ being the
number of (non-zero) projections summed to obtain $\tilde S_\Phi$. This trace can be computed in any orthonormal
basis, so
\[
\mathop{\mathrm{tr}}[\tilde S_\Phi]=\sum_{i=1}^n \mathcal R_\Phi(e_i)= N.
\]
Conversely, we recall  that the version of
Gleason's theorem stated in Corollary~\ref{cor:Gleason} yields that any $f$ satisfying the summation
condition extends to the quadratic form
of an operator $T$. Moreover, $T$ is invertible because $f$ is assumed to be strictly positive
on $\mathbb S$, and
thus $T$ is the frame operator for some $\Phi$. Again invoking \cite{DFKLOW}, the frame can be
assumed to be unit norm, thus $T=\tilde S_\Phi= S_\Phi$.

The equivalence of (i') and (ii') follows from Lemma \ref{lemma:equniformred}.
\end{proof}

We remark that the hypothesis $n \ge 3$ is only relevant for the equivalence with (iii), since
the proof of these equivalences exploits Gleason's theorem.

We further remark that the previous theorem is in fact constructive, since the Jordan-von Neumann theorem
together with the polarization identity can be used to explicitly compute
the frame operator $\tilde{S}_\Phi$ associated with a function  $f : \SSn \to \RR_0^+$ which
extends to a quadratic form. Then, using the technique in \cite{DFKLOW,CCHKP09}, an associated
unit-norm frame can be explicitly constructed.

It is not clear to us, what the correct equivalent condition would be, if we drop the word
`unit-norm' in (i'). This concerns the question of a characterization of normalized frames which
come from tight frames. This in turn is closely related to the still open question of when the
frame vectors can be scaled so that a tight frame is generated as well as to the question of
which equivalence classes contain a tight frame.

%\subsection{Basic Properties}

%Intuitively, $\cR_\Phi(x)$ measures the local redundancy of vectors which are `close' to $x$.
%Draw image of $\RR^2$ with two vectors close to each other. The upper redundancy is almost 2.

%\begin{lemma}
%\label{lemma:help1}
%Let $\Phi = \{\varphi_i\}_{i=1}^N$ be a frame in $\cHn$,
%let $x \in \cHn$, $\norm{x}=1$, and let $i_0 \in \{1,\ldots,N\}$. Then,
%\[
%\cR(\Phi \setminus \{f_{i_0}\})(x) \in [\cR_\Phi(x)- 1,\cR_\Phi(x)].
%\]
%\end{lemma}

%\begin{proof}
%We have
%\[
%\cR_\Phi(x)
%= \sum_{i=1, i \neq i_0}^N \frac{\absip{x}{\varphi_i}^2}{\norm{\varphi_i}^2} + \frac{\absip{x}{f_{i_0}}^2}{\norm{f_{i_0}}^2}
%\le \cR(\Phi \setminus \{f_{i_0}\})(x) + 1.
%\]
%The other estimate is immediate.
%\end{proof}

%************************************************************************************

\section{Upper and Lower Redundancy}

Having studied and characterized redundancy functions, we now focus on the notion
of redundancy itself and will provide more insight into it, in addition to Theorem \ref{theo:desiderata}.

When introducing a new notion, one of the first questions should concern its range. The
following result will provide precise information about the range of the upper and lower
redundancy, and even characterize when there does exist a frame such that a particular pair of values for upper and lower redundancy can be attained.

\bgb{
To avoid inessential complications, we again exclude the case of frames which contain zero vectors.}

\begin{theorem}
\label{lemma:rangeR}
\bgb{Let $\Phi = (\varphi_i)_{i=1}^N$ be a frame for a real or complex Hilbert space $\cHn$
having dimension
$n\ge 2$ and let $\varphi_i \neq 0$ for all $i \in \{1,2, \dots, N\}$.
The upper and lower redundancies of $\Phi$ then}
satisfy the inequalities
\beq \label{eq:firstclaim}
0 < \cR^-_\Phi \le \frac{N}{n} \le \cR^+_\Phi < N.
\eeq
Moreover, if $\cR^-_\Phi = \frac{N}{n}$ or $\cR^+_\Phi = \frac{N}{n}$, then the normalized version of $\Phi$ is a tight frame.

Finally, let $n\le N$, $r_1 \in (0,\frac{N}{n}]$, and $r_2 \in [\frac{N}{n},N)$.
Then the following conditions are equivalent.
\bitem
\item[{\rm (i)}]  There exists a frame $\Phi = (\varphi_i)_{i=1}^N$  for  $\cHn$, $n \ge 2$, such that
\[
\cR^-_\Phi = r_1 \quad \mbox{and} \quad \cR^+_\Phi = r_2.
\]
\item[{\rm (ii)}] We have
\[
(n-1)r_1+r_2 \le N.
\]
\eitem
In particular, for every $r_1\in (0,\frac{N}{n}]$
and every $r_2 \in [\frac{N}{n},N)$, we can find unit-norm frames $\Phi = (\varphi_i)_{i=1}^N$  and
$\Psi = (\psi_i)_{i=1}^N$  with
\[
\cR^-_\Phi = r_1 \quad \mbox{and} \quad \cR^+_\Psi = r_2.
\]
\end{theorem}

\begin{proof}
For the proof of \eqref{eq:firstclaim}, we recall from the proof of Theorem~\ref{theo:equivalenceR} that $N/n$ is the mean value
of $f$ with respect to the probability measure $\omega$ on the sphere $\mathbb S$, which implies
\begin{equation} \label{eq:claimpart1}
 \min_{x \in \mathbb S} \mathcal R_\Phi (x) \le \int_{\mathbb S}  \mathcal R_\Phi (x) d\omega(x)
 = \frac{N}{n} \le \max_{x \in \mathbb S} \mathcal R_\Phi (x) .
\end{equation}
Furthermore, since the vectors in $\Phi$ span the finite dimensional space $\cHn$
for each $x$, there exists some $i \in \{1,\ldots,N\}$ such that $\ip{x}{\varphi_i} \neq 0$,
and hence $\|P_{\langle \varphi_i \rangle} (x) \|^2>0$. This yields
\beq \label{eq:claimpart2}
0<\cR^-_\Phi.
\eeq
Finally, we have
\beq \label{eq:claimhelp}
\cR_\Phi(x) = \sum_{i=1}^N \| P_{\langle \varphi_i \rangle} (x) \|^2
\le \sum_{i=1}^N \norm{x}^2  = N \quad \mbox{for all } x \in \mathbb S.
\eeq
Now assume that we have equality in \eqref{eq:claimhelp} for some $x \in \mathbb S$. This implies that
\[
\| P_{\langle \varphi_i \rangle} (x)  \|^2=\langle P_{\langle \varphi_i \rangle} (x) ,x\rangle=\|x\|^2
\quad \mbox{for all }i \in \{1, 2, \dots n\},
\]
and thus $x$ is an eigenvector of eigenvalue one for all $P_{\langle \varphi_i \rangle}$. Since
each $P_{\langle \varphi_i \rangle}$ is rank one and projects on the span of $\varphi_i$, either $x=0$
or all $\varphi_i$ are collinear. However, if all $\varphi_i$ are collinear, they cannot span
$\mathcal H$ if its dimension is $n\ge 2$. Hence, it follows that
\beq \label{eq:claimpart3}
\cR^+_\Phi < N.
\eeq
Combining \eqref{eq:claimpart1}, \eqref{eq:claimpart2}, and \eqref{eq:claimpart3} proves \eqref{eq:firstclaim}.

For the {\em moreover}-part, we notice that if $\cR^-_\Phi = \frac{N}{n}$, then
the average of $\mathcal R_\Phi$ equals its minimum. This implies
$\omega(\{x \in \mathbb S: \mathcal R_\Phi(x) >N/n\})=0$.
Now the continuity of $\mathcal R_\Phi$ ensures that it is constant.

Next we study the equivalence between (i) and (ii). To prove (i) $\Rightarrow$ (ii),
let $\Phi = (\varphi_i)_{i=1}^N$ be a frame for $\cHn$ which without loss of generality
we can assume to be equal-norm. Then assume that the
frame operator for $(\varphi_i/\|\varphi_i\|)_{i=1}^N$ has eigenvalues
\[
\cR^-_\Phi= r_1 = \lambda_1 \le \lambda_2 \le \ldots \le \lambda_n = r_2= \cR^+_\Phi.
\]
From this,
\[
(n-1)r_1 +r_2 \le \sum_{j=1}^n \lambda_i = N,
\]
hence (ii) follows directly.

For the converse direction, we observe that (ii) implies the existence of real numbers
$r_1 = \lambda_1 \le \lambda_2 \le \ldots \le \lambda_n = r_2$ satisfying
\[
\sum_{j=1}^n \lambda_j = N.
\]
By Proposition \ref{prop:operatorcharac}, we can find an equal-norm frame $(\varphi_i)_{i=1}^N$
whose frame operator has eigenvalues $\lambda_1, \ldots, \lambda_n$. Since
\[
\sum_{i=1}^N \|\varphi_i\|^2 = \sum_{j=1}^n \lambda_j = N,
\]
the frame $(\varphi_i)_{i=1}^N$ is even unit-norm. This proves (i).

It remains to prove the {\em in particular}-part. On the one hand, given $r_1 \in (0,\frac{N}{n}]$,
we choose $r_2 = \frac{N}{n}$. Hence (ii) is satisfied, and by the equivalence of (i)
and (ii) there exists a frame $\Phi = (\varphi_i)_{i=1}^N$ for $\cHn$ so that $\cR^-_\Phi=r_1$.
If, on the other hand, we are given $r_2 \in [\frac{N}{n},n)$, then we may choose $r_1 \in (0,\frac{N}{n}]$
small enough such that (i) is satisfied. Again, arguing as before, there then exists a frame
$\Phi = (\varphi_i)_{i=1}^N$ for $\cHn$ so that $\cR^+_\Phi=r_2$.

The proof of the theorem is complete.
\end{proof}

We remark that Theorem \ref{lemma:rangeR} is constructive, since already
Proposition \ref{prop:operatorcharac} -- which was employed to show existence of an equal-norm
frame in the equivalence of (i) and (ii) -- was constructive.

Let us now, for a moment, analyze the previous result in light of the interpretation of $\cR^-_\Phi$
provided by [D6] in terms of partitioning $\Phi$ into spanning sets and of $\cR^+_\Phi$ provided by [D7]
in terms of partitioning $\Phi$ into linearly independent sets. If the redundancy is not uniform, both
partitions might not coincide. However, if the redundancy is uniform, hence the values of
$\cR^-_\Phi$ and $\cR^+_\Phi$ both equal $N/n$, the partitions suddenly can be chosen to be the same.
In fact, \cite{BCPS09} shows that in this case we can partition our frame into $\lfloor \frac{N}{n}
\rfloor$ linearly independent spanning sets plus a linearly independent set.
We remind the reader that Examples \ref{example:F1}, \ref{example:F2}, and \ref{example:F3}
already gave a hint of the fact that [D6] and [D7] become sharper as they approach the
value $N/n$.

\medskip

In case of an equal-norm frame, the upper and lower redundancy is immediately
computed from the frame bounds.

\begin{lemma}
\label{lemma:equniformred}
Let  $\Phi = (\varphi_i)_{i=1}^N$ be an equal-norm frame for a Hilbert space $\cHn$,
having frame bounds $A$ and $B$. Set $c = \|\varphi_i\|^2$ for all $i=1,\ldots,N$. Then
\[
\cR^-_\Phi = \frac{A}{c} \quad \mbox{and} \quad \cR^+_\Phi = \frac{B}{c}.
\]
\end{lemma}

\begin{proof}
By definition,
\[
\cR_\Phi(x) = \sum_{i=1}^N \| P_{\langle \varphi_i \rangle}(x)\|^2 = c^{-1} \sum_{i=1}^N \absip{x}{\varphi_i}^2.
\]
The claim now follows from the characterization of the frame bounds
\[
A = \min_{x \in \SSn} \sum_{i=1}^N \absip{x}{\varphi_i}^2 \quad \mbox{and} \quad
B = \max_{x \in \SSn} \sum_{i=1}^N \absip{x}{\varphi_i}^2.
\]
\end{proof}

Another question concerns the change of redundancy once an invertible operator is
applied to a frame. This, in particular, relates the upper and lower redundancies of a frame to
those of its canonical dual.

\begin{lemma}
Let $\Phi = (\varphi_i)_{i=1}^N$ be a frame for a real or complex Hilbert space $\cHn$. For any invertible operator $T$ on $\cHn$,
\beq \label{eq:unitaryestimate}
\kappa(T)^{-2}\cR^\pm_\Phi \le \cR^\pm_{T(\Phi)} \le \kappa(T)^2
\cR^\pm_\Phi \, ,
\eeq
where $\kappa(T)=\|T\|\|T^{-1}\|$ denotes the condition number of $T$.

In particular, if $S_\Phi$ denotes the frame operator associated with $\Phi$,
and $\widetilde{\Phi}$ denotes the canonical dual frame of $\Phi$, then
\[
\kappa(S_\Phi)^{-1}
\cR^\pm_\Phi \le \cR^\pm_{\widetilde{\Phi}} \le\kappa(S_\Phi) \cR^\pm_\Phi.
\]
\end{lemma}

\begin{proof}
For each $x \in \SSn$, assuming without loss of generality that $\varphi_i \neq 0$ for all $i$,
\[
\cR_{T(\Phi)}(x) = \sum_{i=1}^N \| P_{\langle T(\varphi_i) \rangle}(x)\|^2
\le \sum_{i=1}^N \frac{\norm{T}^2 \absip{x}{\varphi_i}^2}{\norm{\norm{T^{-1}}^{-1}\varphi_i}^2}
= \norm{T}^{2} \norm{T^{-1}}^{2} \cR_\Phi(x).
\]
Now maximizing or minimizing over $x \in \mathbb S$
implies $\cR^\pm_{T(\Phi)} \le \norm{T}^{2} \norm{T^{-1}}^{2}\cR^\pm_\Phi$. The lower bound
$\norm{T}^{-2} \norm{T^{-1}}^{-2}\cR^\pm_\Phi \le \cR^\pm_{T(\Phi)}$
follows similarly.

The {\em in particular}-part follows by recalling that $\widetilde{\Phi} = S_\Phi^{-\frac12} \Phi$,
by the identity $\|S_\Phi^{\pm 1 /2}\|=\|S_\Phi^{\pm 1}\|^{1/2}$, and by
applying \eqref{eq:unitaryestimate}.
\end{proof}

%Before proving the list of desiderata stated in Subsection \ref{subsec:desiderata},
%we require the following technical lemma.

%\begin{corollary}
%\label{lemma:help2}
%Let $\Phi = \{\varphi_i\}_{i=1}^N$ be a set of vectors in  $\cHn$. Then
%the following conditions are equivalent.
%\bitem
%\item[{\rm (i)}] $\Phi$ forms a frame for $\cHn$.
%\item[{\rm (ii)}] We have $\cR^-_\Phi > 0$.
%\eitem
%\end{corollary}

%\begin{proof}
%(i) $\Rightarrow$ (ii). This Theorem \ref{lemma:rangeR}.
%
%(ii) $\Rightarrow$ (i). Towards a contradiction, we assume that $\Phi$ does not form a frame for $\cHn$. Hence also the
%set $\{\frac{\varphi_i}{\norm{\varphi_i}} : i=1, \ldots, N\}$ does not form a frame for $\cHn$. Since we are
%in the finite-dimensional setting, this implies that the lower frame bound equals zero. By definition
%of $\cR^-_\Phi$, we therefore obtain $\cR^-_\Phi = 0$, a contradiction.
%\end{proof}

\section{Proof of Theorem~\ref{theo:desiderata}}
\label{sec:proof}

\noindent[D\ref{DEE}] If $\Phi = (\varphi_i)_{i=1}^N$ is an equal-norm
tight frame for an $n$-dimensional real or complex Hilbert space $\cHn$, then
\[
\cR_\Phi(x) = \sum_{i=1}^N \| P_{\langle \varphi_i \rangle}(x)\|^2
= \norm{\varphi_1}^{-2} \sum_{i=1}^N \absip{x}{\varphi_i}^2 = \norm{\varphi_1}^{-2}
= \frac{N}{n}.
\]\\[1ex]
[D\ref{DNyquist}] (i) $\Leftrightarrow$ (ii). This follows immediately from Lemma \ref{lemma:equniformred}.\\
(i') $\Rightarrow$ (ii'). Towards a contradiction, assume that
$\Phi$ is not orthogonal. Without loss of generality, $\varphi_1 \not\perp \langle \varphi_2, \ldots, \varphi_N\rangle$,
in particular, $\varphi_1 \neq 0$.
Hence, choosing $x= \varphi_1/\norm{\varphi_1}$, we obtain
\[
\cR_\Phi(x) = 1 + \sum_{i=2}^N \norm{\varphi_1}^{-2} \| P_{\langle \varphi_i \rangle}(x)\|^2 > 1.
\]
Thus $\cR^+_\Phi > 1$, a contradiction to (i').\\
(ii') $\Rightarrow$ (i'). Let $x \in \SSn$. Since $(\varphi_i/\norm{\varphi_i})_{i=1}^N$
is an orthonormal basis, we obtain
\[
\cR_\Phi(x) = \sum_{i=1}^N \| P_{\langle \varphi_i \rangle}(x)\|^2
= \norm{x}^2 = 1.
\]
This implies
\[
\cR^-_\Phi = \max_{x \in \SSn}\cR_\Phi(x) = 1 = \min_{x \in \SSn}\cR_\Phi(x) = \cR^+_\Phi.
\]\\[1ex]
[D\ref{Duplow}] By definition, $\cR^- \le \cR^+$. Moreover, since $\Phi = (\varphi_i)_{i=1}^N$
is a frame for $\cHn$, also $(\varphi_i/\norm{\varphi_i})_{i=1}^N$ is spanning $\cHn$, where
without loss of generality we assume that $\varphi_i \neq 0$ for all $i$. Hence
$(\varphi_i/\norm{\varphi_i})_{i=1}^N$ forms a frame for $\cHn$, and thus possesses a positive
lower frame bound, i.e.,
\[
\cR^-_\Phi = \inf_{x \in \SSn} \sum_{i=1}^N \| P_{\langle \varphi_i \rangle}(x)\|^2 > 0,
\]
as well as a finite upper frame bound, i.e.,
\[
\cR^+_\Phi = \sup_{x \in \SSn} \sum_{i=1}^N \| P_{\langle \varphi_i \rangle}(x)\|^2 < \infty.
\]\\[1ex]
[D\ref{Dadditiv}] The first claim follows from the basic observation that
\[
\sum_{i=1}^N \| P_{\langle \varphi_i \rangle}(x)\|^2 + \sum_{i=1}^n \| P_{\langle e_i \rangle}(x)\|^2
= \sum_{i=1}^N \| P_{\langle \varphi_i \rangle}(x)\|^2 + 1,
\]
which implies
\[
\cR^\pm_{\Phi\cup(e_i)_{i=1}^n} = \cR^\pm_\Phi + 1.
\]
Next, let $\Phi' = (\varphi_i')_{i=1}^M$. Then, for each $x \in \SSn$, we have
\[
\cR^-_{\Phi\cup\Phi'}(x)
= \sum_{i=1}^N \| P_{\langle \varphi_i \rangle}(x)\|^2 + \sum_{i=1}^M \| P_{\langle \varphi_i' \rangle}(x)\|^2.
\]
Hence
\[
\cR^-_{\Phi\cup\Phi'} = \min_{x \in \SSn} \cR^-_{\Phi\cup\Phi'}(x)
\ge \min_{x \in \SSn} \cR^-_\Phi(x) + \min_{x \in \SSn} \cR^-_{\Phi'}(x) = \cR^-_\Phi + \cR^-_{\Phi'}.
\]
as well as
\[
\cR^+_{\Phi\cup\Phi'} = \max_{x \in \SSn} \cR^+_{\Phi\cup\Phi'}(x)
\le \max_{x \in \SSn} \cR^+_\Phi(x) + \max_{x \in \SSn} \cR^+_{\Phi'}(x) = \cR^+_\Phi + \cR^+_{\Phi'}.
\]
The {\em in particular}-part follows immediately from here.\\[1ex]
[D\ref{DInvariance}] For each $x \in \SSn$,
\[
\cR_{U(\Phi)}(x) = \sum_{i=1}^N \| P_{\langle U(\varphi_i) \rangle}(x)\|^2 = \cR_\Phi(U^*(x)).
\]
Since $\norm{U^*(x)}=1$, we conclude that $\cR^\pm_{U(\Phi)} = \cR^\pm_\Phi$.\\
Invariance under scaling and under permutation of the frame vectors is immediate from
the definition of upper and lower redundancies.\\[1ex]
[D\ref{Daverob}]
Without loss of generality, we can assume that each frame element $\varphi_i$ is non-zero.
Since $\cR^-_\Phi$ is the lower frame bound of the frame $(\varphi_i/\|\varphi_i\|)_{i=1}^N$,
it follows from \cite{BCPS09} that $(\varphi_i/\|\varphi_i\|)_{i=1}^N$ can be partitioned into $\lfloor\cR^-_\Phi\rfloor$
spanning sets. Hence, $(\varphi_i)_{i=1}^n$ can also be partitioned into $\lfloor\cR^-_\Phi\rfloor$ spanning sets.\\
The {\em in particular}-part follows automatically from here.\\[1ex]
[D\ref{Dmaxrob}]
Let $S$ be the frame operator of the frame $(\varphi_i/\|\varphi_i\|)_{i=1}^N$,
where we  assume that $\varphi_i \neq 0$ for all $i$.  Then
$(S^{-1/2}(\varphi_i/\|\varphi_i\|))_{i=1}^N$ is a Parseval frame and
\[
\frac{1}{\cR^+}\Id_{\cHn} \le S^{-1} \le \frac{1}{\cR^-}\Id_{\cHn}.
\]
Hence, for all $i=1,2,\ldots,N$,
\beq \label{eq:helpRadoHorn}
\left\|S^{-1/2}\left(\varphi_i/\|\varphi_i\|\right)\right\|^2 \ge \frac{1}{\cR^+}.
\eeq
Next we check the Rado-Horn condition (see \cite{CKS06}).  For this, let $I\subset \{1,2,\ldots,N\}$,
and let $P$ be the orthogonal projection of $(S^{-1/2}(\varphi_i/\|\varphi_i\|))_{i=1}^N$
onto span $(S^{-1/2}(\varphi_i/\|\varphi_i\|))_{i \in I}$.  Employing the fact that $(S^{-1/2}(\varphi_i/\|\varphi_i\|))_{i=1}^N$
is Parseval as well as the estimate \eqref{eq:helpRadoHorn}, we obtain
\begin{eqnarray*}
\dim \langle S^{-1/2}\left(\varphi_i/\|\varphi_i\|\right) : i\in I \rangle
&=& \sum_{i=1}^N \left\|P(S^{-1/2}\left(\varphi_i/\|\varphi_i\|\right))\right\|^2\\
&\ge& \sum_{i\in I}\left\|P(S^{-1/2}\left(\varphi_i/\|\varphi_i\|\right))\right\|^2\\
&=&  \sum_{i\in I}\left\|S^{-1/2}\left(\varphi_i/\|\varphi_i\|\right)\right\|^2\\
&\ge& \frac{|I|}{\cR^+}.
\end{eqnarray*}
Summarizing,
\beq \label{eq:helpRadoHorn2}
\frac{|I|}{\dim \langle S^{-1/2}(\varphi_i/\|\varphi_i\|) : i\in I\rangle} \le \cR^+.
\eeq
By the Rado-Horn theorem \cite{CKS06}, condition \eqref{eq:helpRadoHorn2} implies that
$(S^{-1/2}(\varphi_i/\|\varphi_i\|))_{i=1}^N$ can be partitioned into $\lceil\cR^+\rceil$ linearly
independent sets.  Since $S^{-1/2}$ is an invertible operator, it follows that
$(\varphi_i/\|\varphi_i\|)_{i=1}^N$ -- and hence also $(\varphi_i)_{i=1}^N$ -- can be
partitioned into $\lceil\cR^+\rceil$ linearly independent sets.\\
%\gkk{DELETED PART HERE.}
%The result on spark follows immediately.

\if 0
\section{Interesting questions to pursue}

\bitem
\item Can we make Theorem \ref{theo:equivalenceR} constructive?
\item {\bf Question related to Rado-Horn:}
Let $\Phi = \{\varphi_i\}_{i=1}^N$ be a frame in $\cHn$, and let $x_0 \in \cHn$. Give an equivalent
`Rado-Horn-like' condition for the following condition:
\bitem
\item There exist maximally $M$ disjoint (can be chosen linearly independent) sets $I_j \subset \{1,\ldots,N\}$, $j=1,\ldots,M$
such that $x_0 \in \spann\{\varphi_i : i \in I_j\}$ for all $j$.
\eitem
We can extend this, so that we have a `ladder' towards Rado-Horn: Let $\cA \subset \cHn$ be a subspace.
Give an equivalent `Rado-Horn-like' condition for the following condition:
\bitem
\item There exist maximally $M$ disjoint (can be chosen linearly independent) sets $I_j \subset \{1,\ldots,N\}$, $j=1,\ldots,M$
such that $\cA \subseteq \spann\{\varphi_i : i \in I_j\}$ for all $j$.
\eitem
Notice that this includes Rado-Horn if $\cA$ is chosen to be $\cH$.
\item
Given a frame, compute the scalings of the frame vectors, which give the closest
frame to being tight?
\item Show, using Vivek's result, that a random frame is almost uniformly redundant.
\eitem
\fi

%************************************************************************************

\end{document}